\documentclass[12pt]{article}
\usepackage{amssymb}
\usepackage{amsthm}
\usepackage{amsmath}
\usepackage{graphicx}
\usepackage{enumerate}

\DeclareMathOperator{\Id}{Id}

\DeclareMathOperator{\Fil}{Fil}

\newtheorem{theorem}{Theorem}

\newtheorem{lemma}[theorem]{Lemma}

\newtheorem{remark}[theorem]{Remark}
\newtheorem{example}[theorem]{Example}
\newtheorem{corollary}[theorem]{Corollary}
\title{Filters and ideals in pseudocomplemented posets}
\author{Ivan~Chajda and Helmut~L\"anger}
\date{}
\begin{document}

\footnotetext{Support of the research of both authors by the Austrian Science Fund (FWF), project I~4579-N, and the Czech Science Foundation (GA\v CR), project 20-09869L, entitled ``The many facets of orthomodularity'', is gratefully acknowledged.}

\maketitle

\begin{abstract}
We study ideals and filters of posets and of pseudocomplemented posets and show a version of the Separation Theorem, known for ideals and filters in lattices and semilattices, within this general setting. We extend the concept of a $*$-ideal already introduced by Rao for pseudocomplemented distributive lattices and by Talukder, Chakraborty and Begum for pseudocomplemented semilattices to pseudocomplemented posets. We derive several important properties of such ideals. Especially, we explain connections between prime filters, ultrafilters, filters satisfying the $*$-condition and dense elements. Finally, we prove a Separation Theorem for $*$-ideals.
\end{abstract}



Pseudocomplemented posets play an important role both in algebra and its application, e.g.\ in non-classical propositional logics. Hence, the study of pseudocomplemented posets which need be neither lattices nor semilattices forms a long standing task in the theory of ordered sets. Pseudocomplemented posets were originally investigated by O.~Frink (\cite F) and also by G.~Birkhoff in his monograph \cite B. These investigations were continued in numerous papers, see e.g.\ the paper \cite V by P.~V.~Venkatanarasimhan and several papers by the authors and J.~Paseka (cf.\ \cite{C08} -- \cite{CLP}).

In order to reveal the structure of pseudocomplemented posets, one often uses several accompanying structures like various kinds of ideals and filters. Up to now, the majority of results on these concepts was formulated for semilattices or distributive semilattices, see e.g.\ the papers \cite F, \cite{NN} and \cite{TCB}. In particular, an important concept of an ideal having a strong relation to pseudocomplementation, the so-called $*$-ideal, was introduced in \cite R (under the name as $\delta$-ideal) for pseudocomplemented distributive lattices. It is a natural question if this concept can be extended also to pseudocomplemented posets and if it is fruitful enough to yield important results in this theory. We develop this part of theory in the present paper and we show that these concepts are connected with the concepts of ultrafilters, prime filters, dense elements etc. Moreover, we prove several variants of a Separation Theorem both for arbitrary posets and for pseudocomplemented ones.

We start with recalling several fundamental concepts.

Let $\mathbf P=(P,\leq)$ be a poset and $a,b\in P$. We define
\begin{align*}
   (a] & :=\{x\in P\mid x\leq a\}, \\
   [a) & :=\{x\in P\mid a\leq x\}, \\
L(a,b) & :=\{x\in P\mid x\leq a,b\}, \\
U(a,b) & :=\{x\in P\mid a,b\leq x\}.
\end{align*}

In the literature there exist various definitions of ideals and filters of posets. For our purposes, we introduce these concepts as follows.

An {\em ideal} of $\mathbf P$ is a non-empty subset $I$ of $P$ satisfying $U(x,y)\cap I\neq\emptyset$ and $(x]\subseteq I$ for all $x,y\in I$. Hence every ideal of $\mathbf P$ is convex and down directed. Let $\Id\mathbf P$ denote the set of all ideals of $\mathbf P$. A {\em filter} of $\mathbf P$ is a non-empty subset $F$ of $P$ satisfying $L(x,y)\cap F\neq\emptyset$ and $[x)\subseteq F$ for all $x,y\in F$. Hence every filter of $\mathbf P$ is convex and up-directed. Let $\Fil\mathbf P$ denote the set of all filters of $\mathbf P$. An ideal $I$ of $\mathbf P$ is called
\begin{itemize}
\item {\em proper} if $I\neq P$,
\item a {\em principal ideal} if there exists some $a\in P$ with $(a]=I$,
\item a {\em prime ideal} if it is proper and $a,b\in P$ and $L(a,b)\subseteq I$ together imply $a\in I$ or $b\in I$.
\end{itemize}
A filter $F$ of $\mathbf P$ is called
\begin{itemize}
\item {\em proper} if $F\neq P$,
\item a {\em principal filter} if there exists some $a\in P$ with $[a)=F$,
\item a {\em prime filter} if it is proper and $a,b\in F$ and $U(a,b)\subseteq F$ together imply $a\in F$ or $b\in F$,
\item an {\em ultrafilter} if it is a maximal proper filter of $\mathbf P$.
\end{itemize}

In order to be able to characterize prime ideals and prime filters we need a preliminary lemma.

\begin{lemma}\label{lem5}
Let $\mathbf P=(P,\leq)$ be a poset and $I\subseteq P$. Then we have
\begin{enumerate}[{\rm(i)}]
\item $I$ is down-directed if and only if $P\setminus I$ is up-directed.
\item The following are equivalent:
\begin{enumerate}[{\rm(a)}]
\item For all $a,b\in I:U(a,b)\cap I\neq\emptyset$.
\item For all $a,b\in P$ $(U(a,b)\subseteq P\setminus I$ implies $a\in P\setminus I$ or $b\in P\setminus I)$.
\end{enumerate}
\item The following are equivalent:
\begin{enumerate}
\item[{\rm(c)}] For all $a,b\in P$ $(L(a,b)\subseteq I$ implies $a\in I$ or $b\in I)$.
\item[{\rm(d)}] For all $a,b\in P\setminus I:L(a,b)\cap(P\setminus I)\neq\emptyset$.
\end{enumerate}
\end{enumerate}
\end{lemma}

\begin{proof}
\
\begin{enumerate}[(i)]
\item This is easy to verify.
\item The following are equivalent:
\begin{align*}
& \text{For all }a,b\in I:U(a,b)\cap I\neq\emptyset. \\
& \text{For all }a,b\in P\big(a,b\notin P\setminus I\text{ implies }U(a,b)\not\subseteq P\setminus I\big). \\
& \text{For all }a,b\in P\big(U(a,b)\subseteq P\setminus I\text{ implies }a\in P\setminus I\text{ or }b\in P\setminus I\big).
\end{align*}
\item The following are equivalent:
\begin{align*}
& \text{For all }a,b\in P\big(L(a,b)\subseteq I\text{ implies }a\in I\text{ or }b\in I\big). \\
& \text{for all }a,b\in P\big(a,b\notin I\text{ implies }L(a,b)\not\subseteq I\big). \\
& \text{For all }a,b\in P\setminus I:L(a,b)\cap(P\setminus I)\neq\emptyset.
\end{align*}
\end{enumerate}
\end{proof}

The following two corollaries are immediate consequences of Lemma~\ref{lem5}.

\begin{corollary}\label{cor2}
Let $\mathbf P=(P,\leq)$ be a poset and $I$ an ideal of $\mathbf P$. Then the following are equivalent:
\begin{enumerate}[{\rm(i)}]
\item $I$ is a prime ideal of $\mathbf P$.
\item $P\setminus I$ is a prime filter of $\mathbf P$.
\item $P\setminus I$ is a filter of $\mathbf P$.
\end{enumerate}
\end{corollary}

\begin{proof}
$\text{}$ \\
(i) $\Rightarrow$ (ii): \\
This follows from Lemma~\ref{lem5}. \\
(ii) $\Rightarrow$ (iii): \\
This is trivial. \\
(iii) $\Rightarrow$ (i): \\
This follows from (i) and (iii) of Lemma~\ref{lem5}.
\end{proof}

By duality we obtain
 
\begin{corollary}\label{cor1}
Let $\mathbf P=(P,\leq)$ be a poset and $F$ a filter of $\mathbf P$. Then the following are equivalent:
\begin{enumerate}[{\rm(i)}]
\item $F$ is a prime filter of $\mathbf P$.
\item $P\setminus F$ is a prime ideal of $\mathbf P$.
\item $P\setminus F$ is an ideal of $\mathbf P$.
\end{enumerate}
\end{corollary}

Clearly, the mapping $I\mapsto P\setminus I$ is a bijection from the set of all prime ideals of $\mathbf P$ to the set of all prime filters of $\mathbf P$.

Using previous results, we are able to prove the Separation Theorem well-known for pseudocomplemented lattices and semilattices also for arbitrary posets and also for pseudocomplemented posets under the assumption that the corresponding filter is a prime filter.

\begin{corollary}\label{cor3}
{\rm(}{\bf Separation Theorem for posets}{\rm)} Let $\mathbf P=(P,\leq)$ be a poset, $I$ an ideal of $\mathbf P$ and $F$ a prime filter of $\mathbf P$ and assume $I\cap F=\emptyset$. Then there exists some prime ideal $J$ of $\mathbf P$ with $I\subseteq J$ and $J\cap F=\emptyset$.
\end{corollary}

\begin{proof}
Put $J:=P\setminus F$. Then, since $F$ is a prime filter of $\mathbf P$, we conclude that $J$ is a prime ideal of $\mathbf P$ according to Corollary~\ref{cor1}. Of course, $I\subseteq J$ and $J\cap F=\emptyset$.
\end{proof}

\begin{example}\label{ex2}
The poset $\mathbf P=(P,\leq)$ depicted in Figure~1

\vspace*{-3mm}

\begin{center}
\setlength{\unitlength}{7mm}
\begin{picture}(6,10)
\put(3,1){\circle*{.3}}
\put(2,2){\circle*{.3}}
\put(1,3){\circle*{.3}}
\put(5,3){\circle*{.3}}
\put(1,7){\circle*{.3}}
\put(5,7){\circle*{.3}}
\put(3,9){\circle*{.3}}
\put(3,1){\line(-1,1)2}
\put(3,1){\line(1,1)2}
\put(1,3){\line(0,1)4}
\put(1,3){\line(1,1)4}
\put(5,3){\line(-1,1)4}
\put(5,3){\line(0,1)4}
\put(3,9){\line(-1,-1)2}
\put(3,9){\line(1,-1)2}
\put(2.85,.25){$0$}
\put(1.3,1.85){$a$}
\put(.3,2.85){$b$}
\put(5.4,2.85){$c$}
\put(.3,6.85){$d$}
\put(5.4,6.85){$e$}
\put(2.85,9.4){$1$}
\put(2.25,-.75){{\rm Fig.~1}}
\end{picture}
\end{center}

\vspace*{3mm}

is not a lattice. Since $P$ is finite, every ideal and every filter of $\mathbf P$ is principal. We have:

Ideals of $\mathbf P$ are $(0]$, $(a]$, $(b]$, $(c]$, $(d]$, $(e]$ and $(1]$, \\
filters of $\mathbf P$ are $[0)$, $[a)$, $[b)$, $[c)$, $[d)$, $[e)$ and $[1)$, \\
prime ideals of $\mathbf P$ are $(b]$, $(c]$, $(d]$ and $(e]$, \\
prime filters of $\mathbf P$ are $[a)$, $[c)$, $[d)$ and $[e)$ and \\
ultrafilters of $\mathbf P$ are $[a)$ and $[c)$.

We have: $(a]$ is an ideal of $\mathbf P$ that is not a prime ideal of $\mathbf P$, $[c)$ is a prime filter of $\mathbf P$ and $(a]\cap[c)=\emptyset$. According to Corollary~\ref{cor3} there exists some prime ideal $J$ of $\mathbf P$ with $(a]\subseteq J$ and $J\cap[c)=\emptyset$. One can take $J:=(b]$.
\end{example}

\begin{remark}
Let $\mathbf P=(P,\leq)$ be a poset and $a,b\in P$. Then the following are equivalent: $a\leq b$; $(a]\subseteq(b]$; $[b)\subseteq[a)$. Hence, if $\mathbf P$ has only principal ideals then $(\Id\mathbf P,\subseteq)$ is isomorphic to $\mathbf P$, and if $\mathbf P$ has only principal filters then $(\Fil\mathbf P,\subseteq)$ is dually isomorphic to $\mathbf P$. This shows that neither $(\Id\mathbf P,\subseteq)$ nor $(\Fil\mathbf P,\subseteq)$ must be a lattice. It is in accordance with the fact that the intersection of two ideals {\rm(}filters{\rm)} of $\mathbf P$ need not be an ideal {\rm(}a filter{\rm)} of $\mathbf P$. For instance, for the poset $\mathbf P$ from Example~\ref{ex2} we have
\begin{align*}
(d]\cap(e] & =\{0,a,b,c\}\text{ which is not an ideal of }\mathbf P, \\
[b)\cap[c) & =\{d,e,1\}\text{ which is not an filter of }\mathbf P.
\end{align*}
\end{remark}

It is elementary that every ideal of a finite poset is principal. However, we are able to prove also the converse even in a more general setting.

\begin{lemma}\label{lem1}
Let $\mathbf P=(P,\leq)$ be a poset. Then the following are equivalent:
\begin{enumerate}[{\rm(i)}]
\item Every ideal of $\mathbf P$ is principal.
\item $\mathbf P$ satisfies the Ascending Chain Condition.
\end{enumerate}
\end{lemma}

\begin{proof}
$\text{}$ \\
(i) $\Rightarrow$ (ii): \\
Assume that $\mathbf P$ does not satisfy the Ascending Chain Condition. Then in $\mathbf P$ there exists some infinite ascending chain $a_1<a_2<a_3<\cdots$. But then $\bigcup\limits_{n\geq1}(a_n]$ is an ideal of $\mathbf P$ that is not principal. \\
(ii) $\Rightarrow$ (i): \\
Let $I\in\Id\mathbf P$ and let $M$ denote the set of all maximal elements of $I$. Then $M\neq\emptyset$. If there would exist $a,b\in M$ with $a\neq b$ then there would exist some $c\in U(a,b)\cap I$. Because of the maximality of $a$ and $b$ we would obtain $a=c=b$ contradicting $a\neq b$. Hence $M$ is a singleton. Let $d$ denote the unique element of $M$. If $e\in P$ then there exists some $f\in U(d,e)\cap I$. Because of the maximality of $d$ we obtain $e\leq f=d$. This shows $I=(d]$, i.e.\ $I$ is a principal ideal of $\mathbf P$.
\end{proof}

By duality we get the analogous result for filters.

\begin{lemma}
Let $\mathbf P=(P,\leq)$ be a poset. Then the following are equivalent:
\begin{enumerate}[{\rm(i)}]
\item Every filter of $\mathbf P$ is principal.
\item $\mathbf P$ satisfies the Descending Chain Condition.
\end{enumerate}
\end{lemma}

\begin{corollary}
Every ideal and every filter of a finite poset is principal.
\end{corollary}

Let $\mathbf P=(P,\leq,0)$ be a poset with $0$. Then a filter $F$ of $\mathbf P$ is proper if and only if $0\notin F$. Applying Zorn's Lemma yields the existence of at least one ultrafilter of $\mathbf P$. Now let $a\in P$. Then the element $a^*\in P$ is called the {\em pseudocomplement} of $a$ if it is the greatest element $x$ of $P$ satisfying $L(a,x)=\{0\}$. Hence $a^*$ is the pseudocomplement of $a$ if and only if for arbitrary $x\in P$, $L(a,x)=\{0\}$ is equivalent to $x\leq a^*$. The poset $\mathbf P$ is called {\em pseudocomplemented} if every element $x$ of $P$ has a pseudocomplement $x^*$. In such a case we will denote $\mathbf P$ in the form $\mathbf P=(P,\leq,{}^*,0)$. Here $^*$ is a unary operation on $P$, called the {\em pseudocomplementation}. An element $a$ of $P$ is called
\begin{itemize}
\item {\em Boolean} if $a^{**}=a$,
\item {\em dense} if $a^*=0$.
\end{itemize}
Let $B(\mathbf P)$ denote the set of all Boolean elements of $\mathbf P$ and $D(\mathbf P)$ the set of all dense elements of $\mathbf P$.

The concepts introduced before are illuminated by the following example.

\begin{example}\label{ex4}
The poset from Example~\ref{ex2} is a pseudocomplemented poset $\mathbf P=(P,\leq,{}^*,0)$ with the following pseudocomplementation
\[
\begin{array}{c|ccccccc}
 x  & 0 & a & b & c & d & e & 1 \\
\hline
x^* & 1 & c & c & b & 0 & 0 & 0
\end{array}
\]
We have $B(\mathbf P)=\{0,b,c,1\}$ and $D(\mathbf P)=\{d,e,1\}$. Observe that $D(\mathbf P)$ is not a filter of $\mathbf P$ since $c,d\in D(\mathbf P)$, but $L(c,d)\cap D(\mathbf P)=\emptyset$. Moreover, note that the proper ideals $(0]$ and $(a]$ are not prime ideals.
\end{example}

The following example shows a pseudocomplemented posets which has an ideal that is not principal.

\begin{example}\label{ex1}
Let $\mathbb N$ denote the set of all non-negative integers, put $P:=\mathbb N\cup\{\infty\}$ and define a unary operation $^*$ on $P$ as follows:
\[
x^*:=\left\{
\begin{array}{ll}
1 & \text{if }x=0, \\
0 & \text{otherwise}.
\end{array}
\right.
\]
{\rm(}$x\in P${\rm)}. Then $\mathbf P:=(P,\leq,{}^*,0)$ is a pseudocomplemented poset. We have:

Ideals of $\mathbf P$ are $(a]$ with $a\in P$ together with $\mathbb N$ where $\mathbb N$ is not principal, \\
filters of $\mathbf P$ are $[a)$ with $a\in P$, \\
prime ideals of $\mathbf P$ are $(a]$ with $a\in P\setminus\{\infty\}$ together with $\mathbb N$, \\
prime filters of $\mathbf P$ are $[a)$ with $a\in P\setminus\{0\}$, \\
the only ultrafilter of $\mathbf P$ is $[1)$, \\
$B(\mathbf P)=\{0,1\}$ and $D(\mathbf P)=[1)$.

Observe that $D(\mathbf P)$ is a filter of $\mathbf P$ and all proper ideals of $\mathbf P$ are prime ideals.
\end{example}

Example~\ref{ex4} together with Example~\ref{ex2} shows how the Separation Theorem, i.e.\ Corollary~\ref{cor3}, works also for pseudocomplemented posets. Next we show another such example.

\begin{example}\label{ex3}
The poset $\mathbf P=(P,\leq,{}^*,0)$ depicted in Figure~2

\vspace*{-3mm}

\begin{center}
\setlength{\unitlength}{7mm}
\begin{picture}(6,14)
\put(3,1){\circle*{.3}}
\put(1,3){\circle*{.3}}
\put(5,3){\circle*{.3}}
\put(3,5){\circle*{.3}}
\put(1,7){\circle*{.3}}
\put(5,7){\circle*{.3}}
\put(1,11){\circle*{.3}}
\put(5,11){\circle*{.3}}
\put(3,13){\circle*{.3}}
\put(1,3){\line(1,1)4}
\put(1,3){\line(1,-1)2}
\put(5,3){\line(-1,1)4}
\put(5,3){\line(-1,-1)2}
\put(1,7){\line(0,1)4}
\put(1,7){\line(1,1)4}
\put(5,7){\line(-1,1)4}
\put(5,7){\line(0,1)4}
\put(3,13){\line(-1,-1)2}
\put(3,13){\line(1,-1)2}
\put(2.85,.25){$0$}
\put(.3,2.85){$a$}
\put(5.4,2.85){$b$}
\put(3.4,4.85){$c$}
\put(.3,6.85){$d$}
\put(5.4,6.85){$e$}
\put(.3,10.85){$f$}
\put(5.4,10.85){$g$}
\put(2.85,13.4){$1$}
\put(2.25,-.75){{\rm Fig.~2}}
\end{picture}
\end{center}

\vspace*{3mm}

is pseudocomplemented with the following pseudocomplementation
\[
\begin{array}{c|ccccccccc}
 x  & 0 & a & b & c & d & e & f & g & 1 \\
\hline
x^* & 1 & b & a & 0 & 0 & 0 & 0 & 0 & 0
\end{array}
\]
It is not a lattice. Since $P$ is finite, every ideal and every filter of $\mathbf P$ is principal. We have:

Ideals of $\mathbf P$ are $(0]$, $(a]$, $(b]$, $(c]$, $(d]$, $(e]$, $(f]$, $(g]$ and $(1]$, \\
filters of $\mathbf P$ are $[0)$, $[a)$, $[b)$, $[c)$, $[d)$, $[e)$, $[f)$, $[g)$ and $[1)$, \\
prime ideals of $\mathbf P$ are $(a]$, $(b]$, $(d]$, $(e]$, $(f]$ and $(g]$, \\
prime filters of $\mathbf P$ are $[a)$, $[b)$, $[d)$, $[e)$, $[f)$ and $[g)$, \\
ultrafilters of $\mathbf P$ are $[a)$ and $[b)$, \\
$B(\mathbf P)=\{0,a,b,1\}$ and $D(\mathbf P)=[c)$.

We have: $(c]$ is an ideal of $\mathbf P$ that is not a prime ideal of $\mathbf P$, $[e)$ is a prime filter of $\mathbf P$ and $(c]\cap[e)=\emptyset$. According to Corollary~\ref{cor3} there exists some prime ideal $J$ of $\mathbf P$ with $(c]\subseteq J$ and $J\cap[e)=\emptyset$. One can take $J:=(d]$.
\end{example}

Now let $\mathbf P=(P,\leq,{}^*,0)$ be a pseudocomplemented poset and $a,b\in P$. We repeat several well-known facts on pseudocomplemented posets:
\begin{itemize}
\item $a\leq b$ implies $b^*\leq a^*$, $a\leq a^{**}$ and $a^{***}=a^*$.
\item The following are equivalent: $L(a,b)=\{0\}$; $a\leq b^*$; $a^{**}\leq b^*$; $b\leq a^*$; $b^{**}\leq a^*$; $L(a^{**},b)=\{0\}$.
\item $1:=0^*$ is the greatest element of $\mathbf P$ and $1^*=0$.
\item $0,1\in B(\mathbf P)=P^*$ and $1\in D(\mathbf P)$.
\item $B(\mathbf P)\cap D(\mathbf P)=\{1\}$
\item If $a\in D(\mathbf P)$ and $a\leq b$ then $b\in D(\mathbf P)$.
\end{itemize}
For every subset $A$ of $P$ put $A_*:=\{x\in P\mid x^*\in A\}$ and $A^*:=\{x^*\mid x\in A\}$. Some properties of the operator $_*$ are listed in the following lemma.

\begin{lemma}\label{lem4}
Let $\mathbf P=(P,\leq,{}^*,0)$ be a pseudocomplemented poset, $a\in P$, $I$ an ideal of $\mathbf P$, $F$ a filter of $\mathbf P$, $A,B\subseteq P$ and $A_i\subseteq P$ for all $i\in I$. Then the following hold:
\begin{enumerate}[{\rm(i)}]
\item $I=P$ if and only if $I_*=P$, and $F=P$ if and only if $F_*=P$.
\item $F\neq P$ implies $F\cap F_*=\emptyset$.
\item $a\in A_*$ if and only if $a^{**}\in A_*$.
\item $A\subseteq B$ implies $A_*\subseteq B_*$.
\item $(\bigcup\limits_{i\in I}A_i)_*=\bigcup\limits_{i\in I}(A_i)_*$ and $(\bigcap\limits_{i\in I}A_i)_*=\bigcap\limits_{i\in I}(A_i)_*$.
\end{enumerate}
\end{lemma}

\begin{proof}
\
\begin{enumerate}[(i)]
\item The following are equivalent: $I=P$; $1\in I$; $0^*\in I$; $x^*\in I$ for all $x\in P$; $x\in I_*$ for all $x\in P$; $I_*=P$. Moreover, the following are equivalent: $F=P$; $0\in F$; $1^*\in F$; $x^*\in F$ for all $x\in P$; $x\in F^*$ for all $x\in P$; $F_*=P$.
\item Everyone of the following statements implies the next one: $F\cap F_*\neq\emptyset$; there exists some $b\in F\cap F_*$; $b,b^*\in F$; $\{0\}\cap F=L(b,b^*)\cap F\neq\emptyset$; $0\in F$; $F=P$.
\item The following are equivalent: $a\in A_*$; $a^*\in A$; $a^{***}\in A$; $a^{**}\in A_*$.
\item If $A\subseteq B$ then everyone of the following statements implies the next one: $a\in A_*$; $a^*\in A$; $a^*\in B$; $a\in B_*$.
\item The following are equivalent: $a\in(\bigcup\limits_{i\in I}A_i)_*$; $a^*\in\bigcup\limits_{i\in I}A_i$; there exists some $i\in I$ with $a^*\in A_i$; there exists some $i\in I$ with $a\in(A_i)_*$; $a\in\bigcup\limits_{i\in I}(A_i)_*$. Moreover, the following are equivalent: $a\in(\bigcap\limits_{i\in I}A_i)_*$; $a^*\in\bigcap\limits_{i\in I}A_i$; $a^*\in A_i$ for all $i\in I$; $a\in(A_i)_*$ for all $i\in I$; $a\in\bigcap\limits_{i\in I}(A_i)_*$.
\end{enumerate}
\end{proof}

The concept of a $*$-ideal was introduced for pseudocomplemented distributive lattices in \cite R and for pseudocomplemented semilattices in \cite{TCB}, see also \cite{NN}. Let us note that it was introduced under the different name $\delta$-ideal. We extend this concept to pseudocomplemented posets as follows.

Let $\mathbf P=(P,\leq,{}^*,0)$ be a pseudocomplemented poset. A {\em $*$-ideal} of $\mathbf P$ is an ideal $I$ of $\mathbf P$ such there exists some filter $F$ of $\mathbf P$ with $F_*=I$.

In the following theorem we characterize $*$-ideals of a pseudocomplemented posets.

\begin{theorem}
Let $\mathbf P=(P,\leq,{}^*,0)$ be a pseudocomplemented poset and $I$ an ideal of $\mathbf P$. Consider the following conditions:
\begin{enumerate}[{\rm(i)}]
\item The ideal $I$ is a $*$-ideal of $\mathbf P$.
\item $I^{**}\subseteq I$
\item If $x\in P$, $y\in I$ and $x^*\leq y$ then $y^*\leq x$.
\end{enumerate}
Then {\rm(i)} implies {\rm(ii)}, and {\rm(ii)} together with {\rm(iii)} implies {\rm(i)}.
\end{theorem}

\begin{proof}
Let $a,b\in P$. \\
(i) $\Rightarrow$ (ii): \\
There exists some $F\in\Fil\mathbf P$ with $F_*=I$. Now everyone of the following statements implies the next one: $a\in I$; $a\in F_*$; $a^*\in F$; $a^{***}\in F$; $a^{**}\in F_*$; $a^{**}\in I$. \\
$\big($(ii) and (iii)$\big)$ $\Rightarrow$ (i): \\
Assume $a,b\in I_*$. Then $a^*,b^*\in I$ and since $I\in\Id\mathbf P$ there exists some $c\in U(a^*,b^*)\cap I$. Because of (iii) we have $c^*\in L(a,b)$. Since $c\in I$ we have $c^{**}\in I$ by (ii) and hence $c^*\in I_*$. Together we obtain $c^*\in L(a,b)\cap I_*$. If $I_*\ni a\leq b$ then $b^*\leq a^*\in I$ and hence, since $I\in\Id\mathbf P$, $b^*\in I$, i.e. $b\in I_*$. This shows $I_*\in\Fil\mathbf P$. Now the following are equivalent: $a\in(I_*)_*$; $a^*\in I_*$; $a^{**}\in I$; $a\in I$. This shows $I=(I_*)_*$ proving that $I$ is a $*$-ideal of $\mathbf P$.
\end{proof}

The following lemma shows how $*$-ideals can be produced.

\begin{lemma}\label{lem3}
Let $\mathbf P=(P,\leq,{}^*,0)$ be a pseudocomplemented poset, $a\in P$ and $F$ be a filter of $\mathbf P$. Then the following hold:
\begin{enumerate}[{\rm(i)}]
\item $F_*$ is a $*$-ideal of $\mathbf P$
\item $(a^*]=[a)_*$ is a $*$-ideal of $\mathbf P$
\end{enumerate}
\end{lemma}

\begin{proof}
\
\begin{enumerate}[(i)]
\item Let $b,c\in F_*$. Because of $0^*=1\in F$ we have $0\in F_*$ and hence $F_*\neq\emptyset$. Since $b^*,c^*\in F$  and $F\in\Fil\mathbf P$ there exists some $d\in L(b^*,c^*)\cap F$. Because of $d\in F$ and $d\leq d^{**}$ we obtain $d^{**}\in F$. Together we get $d^*\in U(b,c)\cap F_*$ and hence $U(b,c)\cap F_*\neq\emptyset$. Finally, $e\leq f\in F_*$ implies $F\ni f^*\leq e^*$ and hence $e^*\in F$, i.e.\ $e\in F_*$.
\item For $b\in P$ the following are equivalent: $b\in[a)_*$; $b^*\in[a)$; $a\leq b^*$; $b\leq a^*$; $b\in(a^*]$. The rest follows from (i).
\end{enumerate}
\end{proof}

\begin{example}
The $*$-ideals of the pseudocomplemented poset from Example~\ref{ex4} are $(0]$, $(b]$, $(c]$ and $(1]$, those of the pseudocomplemented poset from Example~\ref{ex1} are $(0]$ and $(\infty]$ and those of the pseudocomplemented poset from Example~\ref{ex3} are $(0]$, $(a]$, $(b]$ and $(1]$.
\end{example}

The next results show which role dense elements play with respect to ideals and filters.

\begin{lemma}\label{lem2}
Let $\mathbf P=(P,\leq,{}^*,0)$ be a pseudocomplemented poset and $I$ a proper $*$-ideal of $\mathbf P$. Then $I$ does not contain a dense element.
\end{lemma}

\begin{proof}
Since $I$ is a $*$-ideal of $\mathbf P$ there exists some $F\in\Fil\mathbf P$ with $F_*=I$. Now everyone of the following statements implies the next one: $I\cap D(\mathbf P)\neq\emptyset$; there exists some $a\in I\cap D(\mathbf P)$; $a\in F_*\cap D(\mathbf P)$; $a\in D(\mathbf P)$ and $a^*\in F$; $0\in F$; $F=P$; $I=P$.
\end{proof}

The assertion of Lemma~\ref{lem2} holds for arbitrary proper ideals of $\mathbf P$ provided $1$ is the only dense element of $\mathbf P$, see (iv) and (v) of the next theorem.

\begin{theorem}\label{th1}
Let $\mathbf P=(P,\leq,{}^*,0)$ be a pseudocomplemented poset and $a\in P$. Then the following are equivalent:
\begin{enumerate}[{\rm(i)}]
\item The element $a$ is Boolean.
\item The ideal $(a]$ is a $*$-ideal of $\mathbf P$.
\item Every ideal $I$ of $\mathbf P$ containing $a$ contains $a^{**}$.
\end{enumerate}
Moreover, the following are equivalent:
\begin{enumerate}
\item[{\rm(iv)}] The element $1$ is the only dense element of $P$.
\item[{\rm(v)}] No proper ideal of $\mathbf P$ contains a dense element.
\end{enumerate}
\end{theorem}

\begin{proof}
$\text{}$ \\
(i) $\Rightarrow$ (ii): \\
According to Lemma~\ref{lem3}, $(a]=(a^{**}]=[a^*)_*$. \\
(ii) $\Rightarrow$ (iii): \\
If $a\in I\in\Id\mathbf P$ then there exists some $F\in\Fil\mathbf P$ with $F_*=(a]$ and from this we successively obtain $a\in F_*$; $a^*\in F$; $a^{***}\in F$; $a^{**}\in F_*$; $a^{**}\in(a]$; $a^{**}\in I$. \\
(iii) $\Rightarrow$ (i): \\
Everyone of the following statements implies the next one: $a\in(a]\in\Id\mathbf P$; $a^{**}\in(a]$; $a\leq a^{**}\leq a$; $a^{**}=a$; $a\in B(\mathbf P)$. \\
(iv) $\Rightarrow$ (v): \\
Everyone of the following statements implies the next one: $P\neq I\in\Id\mathbf P$; $1\notin I$; $I\cap D(\mathbf P)=I\cap\{1\}=\emptyset$. \\
(v) $\Rightarrow$ (iv): \\
Everyone of the following statements implies the next one: $a\in D(\mathbf P)$; $(a]\in\Id\mathbf P$ and $(a]\cap D(\mathbf P)\supseteq\{a\}\neq\emptyset$; $(a]=P$; $a=1$. \\
\end{proof}

Let $\mathbf P=(P,\leq,{}^*,0)$ be a pseudocomplemented poset, $I$ an ideal of $\mathbf P$ and $F$ a filter of $\mathbf P$. We say that $I$ satisfies the {\em $*$-condition} if for every $x\in P$ exactly one of $x$ and $x^*$ belong to $I$. Analogously, we proceed with $F$. Obviously, if $I$ is a prime ideal then $I$ satisfies the $*$-condition if and only if $P\setminus I$ does the same. An analogous statement holds for $F$.

\begin{theorem}\label{th3}
Let $\mathbf P=(P,\leq,{}^*,0)$ be a pseudocomplemented poset, $a\in P$, $I$ an ideal of $\mathbf P$ and $F$ a filter of $\mathbf P$.
\begin{enumerate}[{\rm(i)}]
\item Assume $I$ to contain no dense element. Then $a\notin I$ or $a^*\notin I$.
\item Assume $F$ to be a proper filter. Then $a\notin F$ or $a^*\notin F$.
\item Assume $I$ to be a prime ideal containing no dense element. Then $I$ satisfies the $*$-condition.
\item Assume $F$ to be a prime filter containing all dense elements. Then $F$ satisfies the $*$-condition.
\item Assume $I$ to satisfy the $*$-condition. Then $I$ contains no dense element.
\item Assume $F$ to be a proper filter satisfying the $*$-condition. Then $F$ contains all dense elements.
\end{enumerate}
\end{theorem}

\begin{proof}
\
\begin{enumerate}[(i)]
\item Everyone of the following statements implies the next one: $a,a^*\in I$; there exists some $b\in U(a,a^*)\cap I$; $b\in I$ and $b^*\in L(a^*,a^{**})=\{0\}$; $b\in I\cap D(\mathbf P)$; $I\cap D(\mathbf P)\neq\emptyset$.
\item Everyone of the following statements implies the next one: $a,a^*\in F$; $\{0\}\cap F=L(a,a^*)\cap F\neq\emptyset$; $0\in F$; $F=P$.
\item Because of (i), $a,a^*\in I$ is impossible. If $a,a^*\notin I$ then $a,a^*\in P\setminus I\in\Fil\mathbf P$ according to Corollary~\ref{cor2} contradicting (ii).
\item Since $P\setminus F$ is a prime ideal of $\mathbf P$ according to Corollary~\ref{cor1} and $(P\setminus F)\cap D(\mathbf P)=\emptyset$ we obtain according to (iii) that $P\setminus F$ and hence also $F$ satisfies the $*$-condition.
\item $a\in I\cap D(\mathbf P)$ would imply $a^*=0\in I$ and hence $a,a^*\in I$ contradicting the $*$-condition.
\item $a\in D(\mathbf P)\setminus F$ would imply $a^*=0\notin F$ and hence $a,a^*\notin F$ contradicting the $*$-condition.
\end{enumerate}
\end{proof}

\begin{corollary}\label{cor6}
A prime ideal of a pseudocomplemented poset satisfies the $*$-condition if and only if it contains no dense element.
\end{corollary}

\begin{proof}
This follows from (iii) and (v) of Theorem~\ref{th3}.
\end{proof}

The next theorem describes the connections between the $*$-condition for filters and the maximality of them.

\begin{theorem}\label{th2}
Let $\mathbf P=(P,\leq,{}^*,0)$ be a pseudocomplemented poset and $F$ a proper filter of $\mathbf P$. Consider the following conditions:
\begin{enumerate}[{\rm(i)}]
\item $F$ satisfies the $*$-condition.
\item $F$ is an ultrafilter of $\mathbf P$.
\item If $a\in P\setminus F$, $f_1,f_2\in F$ and $a_1,a_2\in P$ and $L(a,f_i)\subseteq(a_i]$ holds for $i=1,2$ then there exists some $f_3\in F$ and some $a_3\in L(a_1,a_2)$ with $L(a,f_3)\subseteq(a_3]$.
\end{enumerate}
Then {\rm(i)} implies {\rm(ii)}, and {\rm(ii)} together with {\rm(iii)} implies {\rm(i)}.
\end{theorem}

\begin{proof}
$\text{}$ \\
(i) $\Rightarrow$ (ii): \\
Assume $F\subset G\in\Fil\mathbf P$. Then there exists some $a\in G\setminus F$. Because of (i) we conclude $a^*\in F$ which implies $a^*\in G$. Since $G\in\Fil\mathbf P$ we have $\{0\}\cap G=L(a,a^*)\cap G\neq\emptyset$ and hence $0\in G$ which implies $G=P$. \\
$\big($(ii) and (iii)$\big)$ $\Rightarrow$ (i): \\
Let $a\in P\setminus F$. Put
\[
G:=\{x\in P\mid\text{there exists some }f\in F\text{ with }L(a,f)\subseteq(x]\}.
\]
Because of (iii) we have $F\subset F\cup\{a\}\subseteq G\in\Fil\mathbf P$. Since $F$ is an ultrafilter of $\mathbf P$ we conclude $G=P$. Hence $0\in G$ and therefore there exists some $f\in F$ with $L(a,f)\subseteq(0]$. This means $L(a,f)=\{0\}$ and hence $f\leq a^*$ which implies $a^*\in F$. This shows that $a\in F$ or $a^*\in F$. Together with (ii) of Theorem~\ref{th3} we get (i).
\end{proof}

\begin{remark}
If $\mathbf P$ is a meet-semilattice then condition {\rm(iii)} of Theorem~\ref{th2} is trivially satisfied: One may take $f_3:=f_1\wedge f_2$ and $a_3:=a_1\wedge a_2$.
\end{remark}

\begin{corollary}
Let $\mathbf P=(P,\leq,{}^*,0)$ be a pseudocomplemented poset and $F$ a prime filter of $\mathbf P$ containing all dense elements. Then $F$ is an ultrafilter of $\mathbf P$.
\end{corollary}

\begin{proof}
Because of {\rm(iv)} of Theorem~\ref{th3}, $F$ satisfies the $*$-condition and hence $F$ is an ultrafilter of $\mathbf P$ according to Theorem~\ref{th2}.
\end{proof}

Note that the ultrafilters of the pseudocomplemented posets from Examples~\ref{ex4}, \ref{ex1} and \ref{ex3} are exactly the prime filters containing all dense elements.

The following lemma is the key result for the next Separation Theorem.

\begin{lemma}\label{lem6}
Let $\mathbf P=(P,\leq,{}^*,0)$ be a pseudocomplemented poset and $F$ a proper filter of $\mathbf P$ satisfying the $*$-condition. Then $F$ is a prime filter of $\mathbf P$ containing all dense elements, $F_*$ is a $*$-ideal of $\mathbf P$ and $F_*=P\setminus F$.
\end{lemma}

\begin{proof}
From (i) of Lemma~\ref{lem3} we see that $F_*$ is a $*$-deal of $\mathbf P$. Because of (ii) of Lemma~\ref{lem4} we have $F\cap F_*=\emptyset$. Since $F$ satisfies the $*$-condition we have $F\cup F_*=P$. Hence $F_*=P\setminus F$: According to Corollary~\ref{cor1}, $F$ is a prime filter of $\mathbf P$ that contains all dense elements because of (vi) of Theorem~\ref{th3}.
\end{proof}

\begin{corollary}
A proper filter of a pseudocomplemented poset satisfies the $*$-condition if and only if it contains all dense elements.
\end{corollary}

\begin{proof}
This follows from (iv) of Theorem~\ref{th3} and from Lemma~\ref{lem6}.
\end{proof}

\begin{corollary}\label{cor4}
{\rm(}{\bf Separation Theorem for $*$-ideals}{\rm)} Let $\mathbf P=(P,\leq,{}^*,0)$ be a pseudocomplemented poset, $I$ an ideal of $\mathbf P$ and $F$ a filter of $\mathbf P$ satisfying the $*$-condition and assume $I\cap F=\emptyset$. Then there exists some $*$-ideal $J$ of $\mathbf P$ with $I\subseteq J$ and $J\cap F=\emptyset$.
\end{corollary}

\begin{proof}
Put $J:=F_*$ and apply Lemma~\ref{lem6}.
\end{proof}

\begin{example}
Consider the pseudocomplemented poset from Example~\ref{ex4}. We have: $(a]$ is an ideal of $\mathbf P$ that is not a $*$-ideal of $\mathbf P$, $[c)$ is a filter of $\mathbf P$ satisfying the $*$-condition and $(a]\cap[c)=\emptyset$. According to Corollary~\ref{cor4} there exists some $*$-ideal $J$ of $\mathbf P$ with $(a]\subseteq J$ and $J\cap[c)=\emptyset$. One can take $J:=(b]$.
\end{example}

{\bf Author contributions} Both authors contributed equally to this manuscript.

{\bf Funding} Open access funding provided by TU Wien (TUW). This study was funded by the Austrian Science Fund (FWF), project I~4579-N, and the Czech Science Foundation (GA\v CR), project 20-09869L.

{\bf Data Availability} Not applicable.

{\bf Declarations}

{\bf Conflicts of interest} The authors declare that they have no conflict of interest.

Authors' addresses:

Ivan Chajda \\
Palack\'y University Olomouc \\
Faculty of Science \\
Department of Algebra and Geometry \\
17.\ listopadu 12 \\
771 46 Olomouc \\
Czech Republic \\
ivan.chajda@upol.cz

Helmut L\"anger \\
TU Wien \\
Faculty of Mathematics and Geoinformation \\
Institute of Discrete Mathematics and Geometry \\
Wiedner Hauptstra\ss e 8-10 \\
1040 Vienna \\
Austria, and \\
Palack\'y University Olomouc \\
Faculty of Science \\
Department of Algebra and Geometry \\
17.\ listopadu 12 \\
771 46 Olomouc \\
Czech Republic \\
helmut.laenger@tuwien.ac.at
\end{document}